\def\version{03/07/2018 -- version 4
\quad
\href{https://arxiv.org/abs/1710.02974}{arXiv:1710.02974};
published version to appear in \emph{Homology,
Homotopy and Applications}}

\documentclass[11pt,a4paper]{amsart}
\usepackage{fullpage}

\usepackage{amssymb,amscd}
\usepackage{autobreak}

\usepackage[scr,scaled=1.1]{rsfso}

\usepackage{leftidx}

\usepackage[hypertexnames=false]{hyperref}

\usepackage{xkvltxp}
\usepackage{fixme}

\usepackage{mathtools}
\usepackage{relsize}

\usepackage{tikz}
\usetikzlibrary{arrows,calc}

\usepackage[all]{xy}

\usepackage[numeric,lite]{amsrefs}

\usepackage{braket}

\usepackage{pigpen}
\def\PO{\text{\pigpenfont R}}
\def\PB{\text{\pigpenfont J}}

\renewcommand{\thefootnote}{\fnsymbol{footnote}}
\long\def\symbolfootnote[#1]#2{\begingroup%
\def\thefootnote{\fnsymbol{footnote}}\footnote[#1]{#2}\endgroup}

\def\baselinestretch{1.20}

\newtheorem{thm}{Theorem}[section]
\newtheorem{lem}[thm]{Lemma}

\newtheorem{cor}[thm]{Corollary}

\theoremstyle{definition}
\newtheorem{rem}[thm]{Remark}
\newtheorem*{rem*}{Remark}

\newtheorem{examps}[thm]{Examples}

\numberwithin{equation}{section}

\def\cf{\emph{cf.}}
\def\eg{\emph{e.g.}}
\def\ie{\emph{i.e.}}
\def\ds{\displaystyle}
\def\:{\colon}
\def\.{\cdot}
\def\o{\circ}
\def\<{\left\langle}
\def\>{\right\rangle}
\def\({\left(}
\def\){\right)}
\def\ph#1{\phantom{#1}}
\def\epsilon{\varepsilon}
\def\phi{\varphi}
\def\emptyset{\varnothing}
\def\subset{\subseteq}
\def\supset{\supseteq}
\def\leq{\leqslant}
\def\geq{\geqslant}
\def\lla{\longleftarrow}
\def\la{\leftarrow}
\def\lra{\longrightarrow}
\def\Lra{\Longrightarrow}
\def\LRA{\quad\Lra\quad}
\def\ra{\rightarrow}
\def\bar#1{\overline{#1}}
\def\hat#1{\widehat{#1}}
\def\tilde#1{\widetilde{#1}}
\def\iso{\cong}
\def\homeq{\simeq}
\DeclareMathOperator{\aug}{\epsilon}
\DeclareMathOperator{\Char}{char}
\DeclareMathOperator{\id}{id}
\DeclareMathOperator{\Id}{Id}
\DeclareMathOperator{\codom}{codom}
\DeclareMathOperator{\dom}{dom}
\DeclareMathOperator{\im}{im}
\DeclareMathOperator{\rank}{rank}
\DeclareMathOperator{\coker}{coker}
\def\B{\mathbb{B}}
\def\C{\mathbb{C}}
\def\CP{\mathbb{C}\mathrm{P}}
\def\CPi{\CP^\infty}
\def\D{\mathbb{D}}
\def\E{\mathrm{E}}
\def\F{\mathbb{F}}
\def\H{\mathbb{H}}
\def\HP{\mathbb{H}\mathrm{P}}
\def\HPi{\HP^\infty}
\def\k{\Bbbk}
\def\m{\mathfrak{m}}
\def\N{\mathbb{N}}
\def\Q{\mathbb{Q}}
\def\R{\mathbb{R}}
\def\RP{\mathbb{R}\mathrm{P}}
\def\RPi{\RP^\infty}
\def\Z{\mathbb{Z}}
\def\ideal{\triangleleft}
\DeclareMathOperator{\codim}{codim}
\DeclareMathOperator{\bideg}{bideg}
\DeclareMathOperator*{\filt}{filt}
\DeclareMathOperator{\inc}{inc}
\DeclareMathOperator{\mult}{mult}
\DeclareMathOperator{\Aut}{Aut}
\DeclareMathOperator{\cone}{C}
\DeclareMathOperator{\Comod}{Comod}
\DeclareMathOperator{\Mod}{Mod}
\DeclareMathOperator{\End}{End}
\DeclareMathOperator{\Ext}{Ext}
\DeclareMathOperator{\Gal}{Gal}
\DeclareMathOperator{\HAQ}{HAQ}
\DeclareMathOperator{\Har}{Har}
\DeclareMathOperator{\Hom}{Hom}
\DeclareMathOperator{\indet}{indet}
\DeclareMathOperator{\Kumm}{Kumm}
\DeclareMathOperator{\Map}{Map}
\DeclareMathOperator{\ord}{ord}
\DeclareMathOperator{\Pic}{Pic}
\DeclareMathOperator{\TAQ}{TAQ}
\DeclareMathOperator{\Tor}{Tor}
\DeclareMathOperator{\Tr}{Tr}
\DeclareMathOperator*{\colim}{colim}
\DeclareMathOperator*{\hocolim}{hocolim}
\DeclareMathOperator*{\holim}{holim}
\def\del{\partial}
\DeclareMathOperator{\ch}{ch}
\def\O{\mathrm{O}}
\def\SO{\mathrm{SO}}
\def\Sp{\mathrm{Sp}}
\def\Spin{\mathrm{Spin}}
\def\Spinc{\Spin^{\mathrm{c}}}
\def\SU{\mathrm{SU}}
\def\String{\mathrm{String}}
\def\Stringc{\mathrm{String}^{\mathrm{c}}}
\def\U{\mathrm{U}}
\def\bp{\mathrm{bp}}
\def\tmf{\mathrm{tmf}}
\DeclareMathOperator{\ann}{ann}
\DeclareMathOperator{\exc}{exc}
\DeclareMathOperator{\len}{length}
\def\ev{\mathrm{ev}}
\def\cof{\mathrm{c}}
\def\undervee#1{\overset{#1}\vee}
\DeclareMathOperator{\cofib}{cofib}
\DeclareMathOperator{\fib}{fib}
\DeclareMathOperator{\Sq}{Sq}
\DeclareMathOperator{\q}{q}
\def\dlQ{\mathrm{Q}}
\def\hdlQ{\hat{\dlQ}}
\def\tdlQ{\tilde{\dlQ}}
\DeclareMathOperator{\quo}{\pi}
\def\tpsi{\tilde{\psi}}
\def\Ainfty{$\mathcal{A}_\infty$ }
\def\Einfty{$\mathcal{E}_\infty$ }
\def\jc{j^\mathrm{c}}
\DeclareMathOperator{\GL}{GL}
\DeclareMathOperator{\SL}{SL}
\def\os#1{{\underline{#1}}}
\def\botimes{\overline{\otimes}}
\DeclareMathOperator{\CTor}{CTor}
\def\Coalg{\mathbf{Coalg}}
\def\Coalgc{\Coalg^{\mathrm{c}}}
\def\AbGp{\mathbf{AbGp}}
\def\Bialg{\mathbf{BiAlg}}
\def\Bialgo{\mathbf{BiAlg}^0}
\def\HAlg{\mathbf{HopfAlg}}
\def\HAlgc{\HAlg^{\mathrm{c}}}
\def\HRngc{\HRng^{\mathrm{c}}}
\def\HRng{\mathbf{HopfRing}}
\def\Top{\mathbf{Top}}
\def\Vect{\mathbf{Vect}}
\DeclareMathOperator{\lin}{lin}
\DeclareMathOperator{\Prim}{\mathsf{Pr}}
\DeclareMathOperator{\Sym}{\mathsf{Sym}}
\DeclareMathOperator{\FBA}{\mathsf{FBA}}
\DeclareMathOperator{\tFBA}{\widetilde{\mathsf{FBA}}}
\DeclareMathOperator{\FHA}{\mathsf{FHA}}
\def\Set{\mathbf{Set}}
\def\sprod{{\text{\rotatebox[origin=c]{180}{\small$\amalg$}}}}
\def\scoprod{\amalg}
\def\QS0{\dlQ S^0}
\def\QSo0{\dlQ_0S^0}
\def\StA{\mathcal{A}}
\def\StE{\mathcal{E}}
\def\Joker{\mathrm{Joker}}
\def\kO{{k\mathrm{O}}}
\def\Sage{\texttt{Sage}}

\newcommand*\circled[1]{\tikz[baseline=(char.base)]{
  \node[shape=circle,draw,inner sep=.5pt] (char) {#1};}}

\title[Iterated doubles of the Joker and realisability]
{Iterated doubles of the Joker and their realisability}
\author{Andrew Baker
}
\date{\version}
\address{
School of Mathematics \& Statistics,
University of Glasgow, Glasgow G12~8QW, Scotland.}
\email{a.baker@maths.gla.ac.uk}
\urladdr{http://www.maths.gla.ac.uk/$\sim$ajb}
\thanks{
I would like to thank the following for
helpful comments and insights: Bob Bruner
and John Rognes (from whom I learnt an
enormous amount about working with the
Steenrod algebra), Don Davis, Peter Eccles
(who showed  me how to use Toda brackets
to construct complexes efficiently and
so initiated the work described) and
Grant Walker. \\
\emph{The mathematics in this paper owes
much to the insights and inspiration of
\textbf{Michael Barratt} and \textbf{Mark
Mahowald} and I would like to dedicate
it to their memory.}
}

\begin{document}

\begin{abstract}
Let $\mathcal{A}(1)^*$ be the subHopf algebra
of the mod~$2$ Steenrod algebra $\mathcal{A}^*$
generated by $\mathrm{Sq}^1$ and $\mathrm{Sq}^2$.
The \emph{Joker} is the cyclic
$\mathcal{A}(1)^*$-module
$\mathcal{A}(1)^*/\mathcal{A}(1)^*\{\mathrm{Sq}^3\}$
which plays a special r\^ole in the study of
$\mathcal{A}(1)^*$-modules. We discuss realisations
of the Joker both as an $\mathcal{A}^*$-module and
as the cohomology of a spectrum. We also consider
analogous $\mathcal{A}(n)^*$-modules for $n\geq2$
and prove realisability results (both stable and
unstable) for $n=2,3$ and non-realisability results
for $n\geq4$.
\end{abstract}

\maketitle

\section*{Introduction}

The cyclic $\StA(1)^*$-module
$\StA(1)^*/\StA(1)^*\{\Sq^3\}$, commonly known
as the \emph{Joker}, was shown by Adams \&
Priddy~\cite{JFA&SBP} to give rise to a torsion
summand in the Picard group of invertible stable
$\StA(1)^*$-modules.

The reader is warned that there are various
commonly encountered normalisations of the
grading and for topological reasons we use
the one where the lowest degree of a nontrivial
element is~$0$. Here is a representation of
the Joker where a vertical line indicates the
action of~$\Sq^1$ and a curved line indicates
the action of~$\Sq^2$.
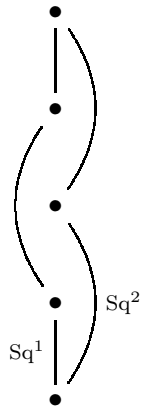
\[
\xymatrix{
& \bullet\ar@{-}[d]\ar@/^15pt/@{-}[dd] & \\
& \bullet\ar@/_15pt/@{-}[dd] & \\
& \bullet\ar@/^15pt/@{-}[dd]^{\Sq^2} & \\
& \bullet\ar@{-}[d]_{\Sq^1} & \\
& \bullet & \\
}
\]

More details on the Joker and its homological
algebra can be found in~\cite{RRB&JPCG}*{appendix~A.8}.
For a recent result which highlights the special
significance of the Joker see~\cite{PB&NR:PicA2}.
Incidentally, the use of the name Joker appears
to be due to Frank Adams, although the earliest
published occurrence that we are aware of is
in~\cite{VG&DJP:MSpin}; it may be based on the
similarity of the diagram above to a traditional
jester's hat.

The $\StA(1)^*$-module structure of the Joker
extends in two ways to an $\StA^*$-module
structure determined by whether $\Sq^4$ acts
non-trivially or not between the bottom and
top degrees. The resulting $\StA^*$-modules
are linear duals of each other. We will show
that both can be realised as cohomology of
finite CW spectra which are Spanier-Whitehead
dual using a construction we learnt from Peter
Eccles, however, it also appeared in Mike Hopkins'
Oxford PhD thesis but seems not to be otherwise
published. We will show in Theorem~\ref{thm:Joker-Unstable}
that these can be realised as cohomology of
spaces with bottom cells of degrees~$2$ and~$4$,
respectively.

We introduce higher versions of the Joker
defined as cyclic $\StA(n)^*$-modules and
show that these can be realised as cohomology
of spectra precisely when~$n\leq3$. Most
cases of the non-realisability result can
be verified by a direct application of
Adams' result on Hopf invariant~$1$,
however, in one case we resort to a more
delicate argument using the precise form
of his factorisation of~$\Sq^{2^r}$ for
$r\geq4$, so we give a proof which applies
in all cases. In the cases where we can
realise these modules, our constructions
depend on the existence of triple Toda
brackets containing the first three
elements of Kervaire invariant~$1$, i.e.,
$\eta^2,\nu^2,\sigma^2$. Finally, we
consider unstable realisations and show
that for $n=1,2$ we can indeed realise
optimal unstable versions of the higher
Joker modules; the techniques used
involve modifying naturally occurring
spaces by mapping into Eilenberg-Mac~Lane
spaces and certain spaces in the spectra~$\kO$
and~$\tmf$, thus giving alternatives to
the stable constructions above.

For the convenience of the reader, we include
some brief appendices in which doubling is
discussed and some connectivity results on
infinite loop spaces are given; this material
is standard but we were unable to locate
convenient references. We  also include
some Adams spectral sequence charts
obtained using Bob Bruner's programmes.

\tableofcontents

\bigskip
\section*{Conventions \& notations}

Throughout we work locally at the prime~$2$.

To avoid excessive display of gradings we
will often suppress cohomological degrees
and write~$V$ for a cohomologically graded
vector space~$V^*$; in particular we will
often write~$\StA$ for the Steenrod algebra.
The linear dual of $V$ is $\mathrm{D}V$ where
$(\mathrm{D}V)^k = \Hom_{\F_2}(V^{-k},\F_2)$,
and we write $V[m]$ for graded vector space
with $(V[m])^k=V^{k-m}$, so for the cohomology
of a spectrum~$X$, $H^*(\Sigma^m X)=H^*(X)[m]$.

For a connected graded algebra $\mathcal{B}^*$
we will write $\mathcal{B}^+$ for its positive
degree part.

\section{$\StA$-module structures on the Joker
and duality}\label{sec:Extensions&duality}

The Joker has two possible $\StA$-module structures
corresponding to the choice of action of~$\Sq^4$
between the top and bottom degrees. The resulting
$\StA$-modules $\Joker_0^*$ and $\Joker_1^*$ are
displayed in the following diagrams in which the
shorter edges represent non-trivial $\Sq^1$ and
$\Sq^2$ actions.
\begin{equation}\label{eq:ExtendedJokers}
\xymatrix{
& \bullet\ar@{-}[d]\ar@/^15pt/@{-}[dd]\ar@{.}@/_28pt/@{.}[dddd]_{\Sq^4=0} & \\
& \bullet\ar@/_15pt/@{-}[dd] & \\
& \bullet\ar@/^15pt/@{-}[dd] & \\
& \bullet\ar@{-}[d] & \\
& \bullet & \\
& {\Joker_0} & \\
}
\quad
\xymatrix{
& \bullet\ar@{-}[d]\ar@/^15pt/@{-}[dd]\ar@/_28pt/@{-}[dddd]_{\Sq^4} & \\
& \bullet\ar@/_15pt/@{-}[dd] & \\
& \bullet\ar@/^15pt/@{-}[dd] & \\
& \bullet\ar@{-}[d] & \\
& \bullet & \\
& {\Joker_1} & \\
}
\end{equation}

Recall that for a left $\StA$-module $M$,
the $\F_2$-linear dual $\mathrm{D}M$ is
naturally a \emph{right} $\StA$-module
where for $f\in\mathrm{D}M$, $\theta\in\StA$
and $x\in M$,
\[
(f\.\theta)(x) = f(\theta x).
\]
There is an associated \emph{left} module
structure given by
\[
(\theta\.f)(x) = (f\.\chi\theta)(x) = f(\chi\theta x),
\]
where $\chi\:\StA\to\StA$ is the antipode.
For a finite CW complex spectrum $Z$, as
a left $\StA$-module the cohomology of
the Spanier-Whitehead dual~$DZ$ satisfies
\[
H^*(DZ) \iso \mathrm{D}(H^*(Z)).
\]

Since the following relations hold in~$\StA$,
\begin{align*}
\Sq^3 &= \Sq^1\Sq^2, \\
\Sq^2\Sq^2 &= \Sq^1\Sq^2\Sq^1, \\
\chi\Sq^1 &= \Sq^1, \\
\chi\Sq^2 &= \Sq^2, \\
\chi\Sq^4
&= \Sq^4 + \Sq^1\Sq^2\Sq^1 = \Sq^4 + \Sq^2\Sq^2,
\end{align*}
it follows that $\Joker_0$ and $\Joker_1$
are dual up to a degree shift, i.e.,
\begin{equation}\label{eq:Jokerduality}
\Joker_1 \iso \mathrm{D}\Joker_0[4].
\end{equation}

\section{Doubling and higher versions
of the Joker}\label{sec:HigherVersions}

An account of doubling can be found in
Margolis~\cite{HRM:Book}*{section~15.3},
but since we need to use an iterated
version we give a detailed account in
Appendix~\ref{app:Doubling}.

For each $n\geq2$, iterated doubling gives
a generalisation of the Joker to a cyclic
$\StA(n)$-module $\Joker(n)=\Delta^{(n-1)}(\Joker)$.
The actions of $\Sq^{2^{n-1}}$ and $\Sq^{2^n}$
on $\Joker(n)$ are shown below.
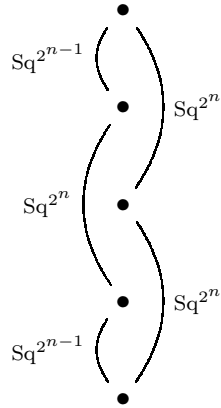
\[
\xymatrix{
& \bullet\ar@/_10pt/@{-}[d]_{\Sq^{2^{n-1}}}\ar@/^15pt/@{-}[dd]^{\Sq^{2^n}} & \\
& \bullet\ar@/_15pt/@{-}[dd]_{\Sq^{2^n}} & \\
& \bullet\ar@/^15pt/@{-}[dd]^{\Sq^{2^n}} & \\
& \bullet\ar@/_10pt/@{-}[d]_{\Sq^{2^{n-1}}} & \\
& \bullet & \\
}
\]

There are two extensions to $\StA$-module
structures, each determined by a choice of
action by $\Sq^{2^{n+1}}$ from the top to
the bottom degree.
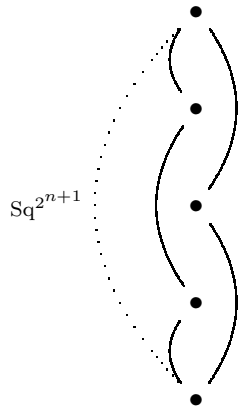
\[
\xymatrix{
& \bullet\ar@/_10pt/@{-}[d]\ar@/^15pt/@{-}[dd] & \\
& \bullet\ar@/_15pt/@{-}[dd] & \\
& \bullet\ar@/^15pt/@{-}[dd] & \\
& \bullet\ar@/_10pt/@{-}[d] & \\
& \bullet\ar@{.}@/^38pt/[uuuu]^{\Sq^{2^{n+1}}} & \\
}
\]
We denote these $\StA$-modules by $\Joker(n)_0$
and $\Joker(n)_1$ depending on whether
$\Sq^{2^{n+1}}$ acts trivially or not.
It is straightforward to verify that
the action of $\chi\Sq^{2^{n+1}}$ on
$\Joker(n)_0$ is non-trivial and
\[
\Joker(n)_1 \iso
  \mathrm{D}\Joker(n)_0[2^{n+1}].
\]

As an $\StA(n)$-module, $\Joker(n)$ is finitely
presented. For example, $\Joker(1)$ has minimal
presentation
\[
0\leftarrow \Joker(1)\leftarrow \StA(1)\leftarrow \StA(1)[3],
\]
while
\[
\Joker(2)
=
\StA(2)/\StA(2)\{\mathrm{P}_1^0,\mathrm{P}_2^0,\mathrm{P}_3^0,\Sq^6\}
=
\StA(2)/\StA(2)\{\mathrm{P}_1^0,\mathrm{P}_2^0,\Sq^6\},
\]
so it has a minimal presentation
\[
0\leftarrow \Joker(2)\leftarrow \StA(2)
\leftarrow\StA(2)[1]\oplus\StA(2)[3]\oplus\StA(2)[6].
\]
Finally,
\[
\Joker(3) =
\StA(3)/\StA(3)\{\mathrm{P}_1^0,\mathrm{P}_2^0,\mathrm{P}_3^0,\mathrm{P}_1^1,\mathrm{P}_2^1,\Sq^{12}\}
=
\StA(3)/\StA(3)\{\mathrm{P}_1^0,\mathrm{P}_1^1,\mathrm{P}_2^1,\Sq^{12}\}
\]
and there is a minimal presentation
\[
0\leftarrow \Joker(3)\leftarrow \StA(3)
\leftarrow
\StA(3)[1]\oplus\StA(3)[2]\oplus\StA(3)[6]\oplus\StA(3)[12].
\]
Of course the $\StA$-modules $\Joker(n)_0$
and $\Joker(n)_1$ are not finitely presented.
Such resolutions can be found obtained using
the~\Sage{} package of Mike Catanzaro and
Bob Bruner
\begin{center}
\href{http://www.math.wayne.edu/~mike/mods/}
{http://www.math.wayne.edu/$\sim$mike/mods/}
\end{center}
documented in~\cite{MC:FPMStA}.

\subsection*{Coactions}
As left $\StA_*$-comodules, $\Joker_0$ and
$\Joker_1$ have coactions which satisfy
\begin{align*}
\Joker_0:\quad \psi x_4
&= 1\otimes x_4 + \zeta_1\otimes x_3
   + \zeta_1^2\otimes x_2 + \zeta_2\otimes x_1
   + \zeta_1\xi_2\otimes x_0, \\
\Joker_1:\quad \psi x_4
&= 1\otimes x_4 + \zeta_1\otimes x_3
   + \zeta_1^2\otimes x_2 + \xi_2\otimes x_1
   + \zeta_1\zeta_2\otimes x_0.
\end{align*}
These differ by the term $\zeta_1^4\otimes x_0$
which detects the~$\Sq^4$ action. Similar
formulae can be obtained for the coactions
in the higher cases using the dual Frobenius
action.

\section{Some recollections on Toda brackets}
\label{sec:TodaBrackets}

For ease of reference, we recall some basic
ideas about triple Toda brackets in homotopy
theory. A classic source for the basic ideas
is the book of Mosher \& Tangora~\cite{Mosher&Tangora}
and Toda's seminal work~\cite{Toda} provides
a more exhaustive account, while
Cohen~\cite{JMC:Decomp} gives a different
treatment, also discussed by Whitehead~\cite{GW:Notes}.

Let
\begin{equation}\label{eq:Toda-1}
W\xrightarrow{f}X\xrightarrow{g}Y\xrightarrow{h}Z
\end{equation}
be a sequence of maps (of based spaces or spectra)
and assume that~$gf$ and~$hg$ are null homotopic.
The mapping sequence for~$g$ extends to a commutative
diagram of solid arrows
\begin{equation}\label{eq:Toda-2}
\xymatrix{
&& \Sigma W\ar[dr]^{\Sigma f}\ar@{-->}[d]_{f_\flat} & \\
X\ar[r]^g & Y\ar[r]^j\ar[dr]_h & C(g)\ar[r]^k\ar@{-->}[d]^{h^\sharp}
  & \Sigma X\ar[r]^{\Sigma g} & \Sigma Y \\
&& Z & \\
}
\end{equation}
and the composition $h^\sharp f_\flat\:\Sigma W\to Z$
represents the Toda bracket $\langle f,g,h\rangle$.
Of course this element is not necessarily well
defined up to homotopy: the choices in $f_\flat$
and $h^\sharp$ contribute indeterminacy subgroups
$h_*[\Sigma W,Y]$ and $(\Sigma f)^*[\Sigma X,Y]$
and when $W$ is a suspension or a spectrum
\[
\indet\langle f,g,h\rangle =
   h_*[\Sigma W,Y]+(\Sigma f)^*[\Sigma X,Y]
\]
and
\[
\langle f,g,h\rangle =
   h^\sharp f_\flat
     + h_*[\Sigma W,Y]+(\Sigma f)^*[\Sigma X,Y]
\]
for some given choice of $f_\flat$ and $h^\sharp$.

Here are some important examples of such
Toda brackets in the stable homotopy groups
of spheres $\pi_*(S)$ where $S=S^0_{(2)}$
is the $2$-local sphere spectrum. As usual,
we identify $\theta\in\pi_n(S)$ with
$\Sigma^k\theta\in\pi_{n+k}(\Sigma^kS)\iso\pi_{n+k}(S^k)$.
\begin{subequations}\label{eq:Todabrackets}
\begin{align}
\langle2,\eta,2\rangle &= \{\eta^2\},
\label{eq:Todabrackets-eta} \\
\langle\eta,\nu,\eta\rangle &= \{\nu^2\},
\label{eq:Todabrackets-nu} \\
\langle\nu,\sigma,\nu\rangle &= \{\sigma^2\}.
\label{eq:Todabrackets-sigma}
\end{align}
\end{subequations}
Of course these elements $\theta_1=\eta^2$,
$\theta_2=\nu^2$ and $\theta_3=\sigma^2$
are the first elements of Kervaire
invariant~$1$.

\begin{proof}[Proof/justification]
Using the Peterson-Stein formula~\cite{Mosher&Tangora}
or calculations with Massey products
in $\Ext_{\StA}$, it is straightforward
to see that these brackets contain the
stated elements; alternatively
see~\cite{Toda}*{corollary~3.7}. Also,
\begin{align*}
\indet\langle2,\eta,2\rangle &= 2\pi_1(S) = 0, \\
\indet\langle\eta,\nu,\eta\rangle &= \eta\pi_5(S) = 0,
\end{align*}
since $\pi_1(S)\iso\Z_2$ and $\pi_5(S) = 0$.
To see that $\indet\langle\nu,\sigma,\nu\rangle=0$,
we need to consider the Adams spectral
sequence in degree~$14$; a part of the
$\mathrm{E}_2$-term is shown in
Figure~\ref{fig:ASS-S0} and although
there is an element
$h_2\mathrm{P}h_2\in\mathrm{E}_2^{6,20}$
this is killed by the differential~$d_3$,
so $\nu\pi_{11}(S)=0$.
\end{proof}

\section{Constructing Joker spectra}
\label{sec:Constructing}

The main idea for this construction was
explained to us by Peter Eccles, and it
also appears in the unpublished Oxford
PhD thesis of Mike Hopkins~\cite{MJH:OxfordPhD}
(see section~1.7). We will make use of
the well-known Toda bracket
$\langle2,\eta,2\rangle=\{\eta^2\}\subseteq\pi_2(S^0)$.

Let $2^\sharp\:C(\eta)\to S^0$ extend~$2$
on the bottom cell; there is no indeterminacy
in this choice because the non-trivial
element $\eta^2\in\pi_2(S^0)$ is in the
image of the map $\pi_2(S^1)\to\pi_2(S^0)$
induced by $\eta\:S^2\to S^1$ on domains.
Let $2_\flat\:S^2\to C(\eta)$ be the
coextension of the degree~$2$ map onto the
top cell; again there is no indeterminacy
in this choice since the non-trivial element
$\eta^2\in\pi_2(S^0)$ is in the image of
the map $\pi_2(S^1)\to\pi_2(S^0)$ induced
by $\eta\:S^1\to S^0$ on codomains.
\[
\xymatrix{
&&
\ar@/_8pt/@{-->}[dll]_{\eta}\ar@/_8pt/@{-->}[dl]^{\eta^2}
S^2\ar[dr]^2\ar[d]^{2_\flat}
&& \\
S^1\ar[r]_\eta & S^0\ar[dr]_2\ar[r]
&\ar[r] C(\eta)\ar[d]_{2^\sharp}
& \ar@/^8pt/@{-->}[dl]_{\eta^2}S^2\ar[r]^\eta
& S^1\ar@/^8pt/@{-->}[dll]^{\eta}\\
&& S^0 & &
}
\]
By definition of $\langle2,\eta,2\rangle$,
the composition $\eta^2-2^\sharp\circ2_\flat$
in the following diagram is null homotopic.
\[
\xymatrix{
S^2 \ar[rr]_(.4){\eta\vee 2_\flat}\ar@/^16pt/[rrrr]^{\eta^2-2^\sharp\circ2_\flat}
&& S^1\vee C(\eta)\ar[rr]_(.6){\eta\vee(-2^\sharp)}
&& S^0 \\
}
\]
The mapping sequences for $\eta\vee 2^\sharp$
and $\eta\vee 2_\flat$ together yield the
following diagram of solid and dashed arrows.

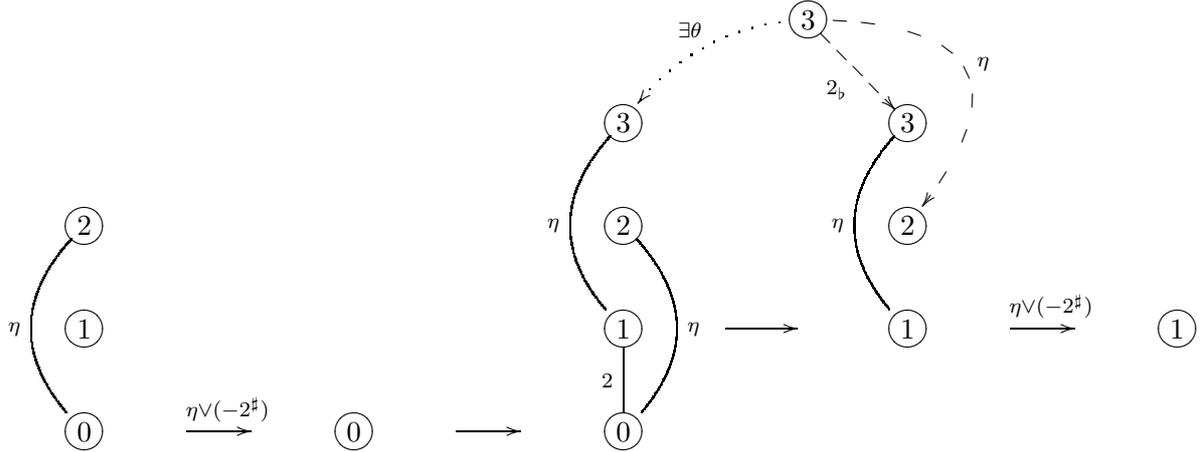
\[
\xymatrix{
 &&& &&& &&
*+[o][F]{3}\ar@/_10pt/@{.>}[lld]_{\exists\theta}
\ar@{-->}[dr]_{2_\flat}
\ar@/^45pt/@{-->}[ddr]^{\eta}
& &&&  \\
&&& &&&*+[o][F]{3}\ar@/_20pt/@{-}[dd]_{\eta}
&&&*+[o][F]{3}\ar@/_20pt/@{-}[dd]_{\eta} &&&  \\
*+[o][F]{2}\ar@/_20pt/@{-}[dd]_{\eta}
&&& &&&*+[o][F]{2}\ar@/^20pt/@{-}[dd]^{\eta}
&&&*+[o][F]{2}&&&  \\
*+[o][F]{1}&&& &&&*+[o][F]{1}\ar@{-}[d]_{2}
&\ar[r]&&*+[o][F]{1}&\ar[r]^(.6){\eta\vee(-2^\sharp)}&&*+[o][F]{1}  \\
*+[o][F]{0}&\ar[r]^(.6){\eta\vee(-2^\sharp)}&&*+[o][F]{0}
&\ar[r]&&*+[o][F]{0} &&& &&& \\
}
\]

\medskip
\noindent
This shows the existence of a map $\theta$
which is well-defined up to indeterminacy
which lies in the image of
\[
\pi_3(S^0)/[\eta\pi_3(S^1)+2\pi_3(S^0)]
= \pi_3(S^0)/2\pi_3(S^0) \iso \Z_2,
\]
so there are two choices of such a map
$\theta$ up to homotopy. Because of the
$\eta$ component on the $2$-sphere, the
mapping cone of $\theta$ has the form
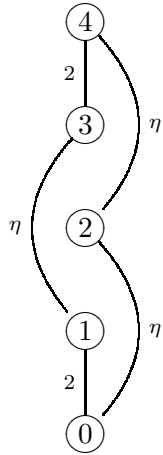
\[
\xymatrix{
&*+[o][F]{4}\ar@/^20pt/@{-}[dd]^{\eta}\ar@{-}[d]_{2}& \\
&*+[o][F]{3}\ar@/_20pt/@{-}[dd]_{\eta}& \\
&*+[o][F]{2}\ar@/^20pt/@{-}[dd]^{\eta}& \\
&*+[o][F]{1}\ar@{-}[d]_{2}& \\
&*+[o][F]{0}&
}
\]
whose cohomology is the Joker $\StA(1)$-module.
The $\StA$-module structure has a $\Sq^4$
action between degrees~$0$ and~$4$ and this
could be zero or non-zero. Each of these
possibilities can occur, depending on which
of the two of choices for~$\theta$ is made.
Putting all this together with the algebraic
identity~\eqref{eq:Jokerduality} we obtain
the following.

\begin{thm}\label{thm:Jokerspectra}
There are two equivalence classes of
finite $2$-local CW spectra $J_0$ and
$J_1$ whose cohomology realise the
$\StA$-modules\/ $\Joker_0$
and\/ $\Joker_1$. Up to suspension,
$J_0$ and $J_1$ are Spanier-Whitehead
dual, i.e.,
\[
DJ_0 \simeq \Sigma^{-4}J_1.
\]
\end{thm}

Low dimensional portions of the Adams
$\mathrm{E}_2$-terms for such Joker spectra
are shown in Figures~\ref{fig:J0}
and~\ref{fig:J1}; there are no differentials
in these regions and up to degree~$12$ they
differ only by an~$h_0$ multiplication in
the $6$-column.

Here is a useful consequence of the existence
of such Joker spectra; we assume this was
known to Mark Mahowald but have not been
able to locate an explicit statement on the
existence of Joker spectra in his published
work -- however, see Remark~\ref{rem:A(1)spectrum}
and also~\cite{MJH:OxfordPhD}*{section~1.7}.
\begin{cor}\label{cor:Jokerspectra}
The $1$-connected cover of $\kO$ satisfies
\[
\kO\wedge\Sigma^2J_0
   \sim \kO\langle2\rangle
   \sim \kO\wedge\Sigma^2J_1.
\]
\end{cor}

It is well known that the other spectra
which appear in the Whitehead tower of
$\kO$ can all be defined in terms of
$\kO$-module spectra of the form $\kO\wedge W$
where the $\StA(1)^*$-module $H^*(W)$
has one of the following forms.
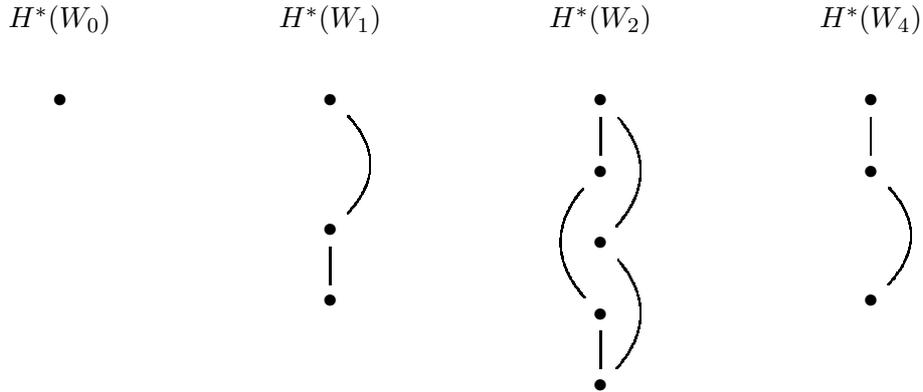
\[
\xymatrix@C=0.5cm@R=0.5cm{
& H^*(W_0) & \\
&\bullet&
}
\quad
\xymatrix@C=0.5cm@R=0.5cm{
& H^*(W_1) & \\
&\bullet\ar@/^15pt/@{-}[dd] & \\
&& \\
&\bullet\ar@{-}[d]& \\
&\bullet &
}
\quad
\xymatrix@C=0.5cm@R=0.5cm{
& H^*(W_2) & \\
&\bullet\ar@/^15pt/@{-}[dd]\ar@{-}[d] & \\
&\bullet\ar@/_15pt/@{-}[dd] &\\
&\bullet\ar@/^15pt/@{-}[dd]& \\
&\bullet\ar@{-}[d] & \\
&\bullet&
}
\quad
\xymatrix@C=0.5cm@R=0.5cm{
& H^*(W_4) & \\
&\bullet \ar@{-}[d]& \\
&\bullet\ar@/^15pt/@{-}[dd] & \\
&& \\
&\bullet&
}
\]
In general, when $r=0,1,2,4$ and $m\geq0$,
\[
\kO\langle8m+r\rangle \sim \kO\wedge\Sigma^{8m}W_r.
\]
For more details see~\cites{MM:bo-res,MM&RJM:Sq4}.
\begin{rem}\label{rem:A(1)spectrum}
A spectrum whose cohomology agrees with $\StA(1)$
as an $\StA(1)$-module (referred to as a `space'
in~\cite{MM&RJM:Sq4}*{remark~1.6}) can be
constructed using our Joker spectra. Starting
with $J$ to be either of $J_0$ or $J_1$ we see
in Figures~\ref{fig:J0} and~\ref{fig:J1} that
there is a generator~$u$ of $\pi_2(J)\iso\Z_{(2)}$
in Adams filtration~$1$ (this is a manifestation
of $v_1$ and has degree~$2$ on the $2$-cell).
Also, $\eta u=0$ so~$u$ extends to a map
$S^2\cup_\eta e^4\to J$. As $\pi_4(J)=0$, this
also extends to a map
$f\:S^2\cup_\eta e^4\cup_2 e^3\to J$ where the
attaching map of the $3$-cell is yet another
avatar of~$v_1$. The cohomology of the mapping
cone~$C(f)$ has basis elements in the same
degrees as $\StA(1)$ and all but one Steenrod
operation (indicated by the dashed line below)
are clear from the above description.
\[
\xymatrix@C=0.5cm@R=0.5cm{
&& \circ & \\
&& \circ\ar@{.}[u]^{\Sq^1} & \\
&& \bullet\ar@{--}@/_15pt/[uu]_{?} & \\
\circ\ar@{.}@/^8pt/[rruu]^{\Sq^2}&&\bullet\ar@{-}[u]& \\
\bullet\ar@{.}[u]\ar@{-}@/^8pt/[rruu]&&& \\
\bullet\ar@{-}@/^8pt/[rruu]&&& \\
\bullet\ar@{-}@/^15pt/[uu]\ar@{-}[u]&&& \\
}
\]
The relation $\Sq^2\Sq^2 = \Sq^1\Sq^2\Sq^1$
shows that this is indeed a $\Sq^2$ and
therefore as $\StA(1)$-modules,
$H^*(C(f))\iso\StA(1)$.

Of course, the action of $\StA$ on $H^*(J)$
extends to one on $H^*(C(f))$, thus giving
at least two different $\StA$-module structures
on~$\StA(1)$.
In~\cite{DMD-MM:v1v2periodicity}*{theorem~1.4},
Davis \& Mahowald gave a different construction
realising all four of the possible $\StA$-module
structures known to exist.
\end{rem}

For small $n$ we can realise $\Joker(n)_0$
and $\Joker(n)_1$ as the cohomology of
spectra.
\begin{thm}\label{thm:Jokerspectra24}
For $n=2,3$ there are finite $2$-local
CW~spectra $J(n)_0$ and $J(n)_1$ whose
cohomology restricts to $\StA(n)$-modules
isomorphic to $\Joker(n)$. These realise
the two $\StA$-module structures extending
the two dual $\StA(n)$-module structures.
Up to suspension, $J(n)_1$ can be taken
to be the Spanier-Whitehead dual of~$J(n)_0$.
\end{thm}
\begin{proof}
The basic approach of Section~\ref{sec:Constructing}
also works using the Toda brackets $\langle\eta,\nu,\eta\rangle$
and $\langle\nu,\sigma,\nu\rangle$ given
in~\eqref{eq:Todabrackets}.

\subsection*{The case $n=2$:}
Using ideas and notation from Section~\ref{sec:TodaBrackets}
we can form a map
\[
S^3\vee C(\Sigma\nu)\xrightarrow{\nu\vee(-\eta^\sharp)}S^0
\]
whose mapping cone fits into a cofibre
sequence
\[
S^0 \to C(\nu\vee(-\eta^\sharp))
\to S^4\vee C(\Sigma^2\nu)\to S^1.
\]
The map
\[
S^7
\xrightarrow{\;\Sigma^4\nu\vee(\Sigma^5\eta)_\flat\;}
S^4\vee C(\Sigma^2\nu)
\]
projects to
\[
(\Sigma\nu^2-\eta^\sharp(\Sigma^5\eta)_\flat)\:S^7 \to S^1
\]
which is null homotopic as
$\langle\eta,\nu,\eta\rangle=\{\nu^2\}$.
Therefore $\Sigma^4\nu\vee(\Sigma^5\eta)_\flat$
factors through a map
\[
\theta'\: S^7 \to C(\nu\vee(-\eta^\sharp))
\]
whose mapping cone has the following form.
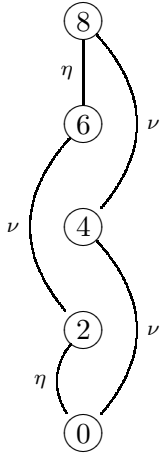
\[
\xymatrix{
&*+[o][F]{8}\ar@/^20pt/@{-}[dd]^{\nu}\ar@{-}[d]_{\eta}& \\
&*+[o][F]{6}\ar@/_20pt/@{-}[dd]_{\nu}& \\
&*+[o][F]{4}\ar@/^20pt/@{-}[dd]^{\nu}& \\
&*+[o][F]{2}\ar@/_10pt/@{-}[d]_{\eta}& \\
&*+[o][F]{0}&
}
\]
The indeterminacy in $\theta'$ is
\[
\pi_7(S^0)/[\nu\pi_4(S^0) + \eta\pi_6(S^0)]
= \pi_7(S^0)/\{0\} \iso \Z_8.
\]

\subsection*{The case $n=3$:} A similar
argument works and we obtain the desired
spectrum as the mapping cone of a map
\[
\theta''\: S^{15} \to C(\sigma\vee(-\nu^\sharp)).
\]
The indeterminacy in $\theta''$ is
\[
\pi_{15}(S^0)/[\sigma\pi_8(S^0) + \nu\pi_{12}(S^0)]
= \pi_{15}(S^0)/\{0\} \iso \Z_{32}.
\qedhere
\]
\end{proof}

Parts of the $\Ext_{\StA}$ charts for the
cases~$n=2,3$ are shown in
Figures~\ref{fig:J(2)0}, \ref{fig:J(2)1},
\ref{fig:J(3)0} and~\ref{fig:J(3)1}.

\begin{rem}\label{rem:DA(1)spectrum}
In similar fashion to the construction of
a realisation of~$\StA(1)$ described in
Remark~\ref{rem:A(1)spectrum}, we can use
either of the spectra $J(2)_0$ or $J(2)_1$
to build a spectrum whose cohomology realises
the double $\Delta\StA(1)$. In~\cite{AM:H*tmf},
such a spectrum is denoted $D\StA(1)$, but
this clashes with standard notation for
Spanier-Whitehead duals so we avoid using
it here. We sketch the details, making use
of the information in Figures~\ref{fig:J(2)0}
and~\ref{fig:J(2)1}.

Choose $J(2)$ to be either $J(2)_0$ or $J(2)_1$.
We start with a map $S^5\to J(2)$ realising
the generator of $\pi_5(J(2))\iso\Z_2$
(this has in Adams filtration~$1$). Since
$\nu\pi_5(J(2))=\{0\}$, this extends to
a map $S^5\cup_\nu e^9\to J(2)$. As
$\pi_{10}(J(2))=\{0\}$ there is an extension
to a map
$g\:S^5\cup_\nu e^9\cup_\eta e^{11}\to J(2)$
and the cohomology of its mapping cone
has the following form where short/long
lines indicate $\Sq^2$/$\Sq^4$ actions.
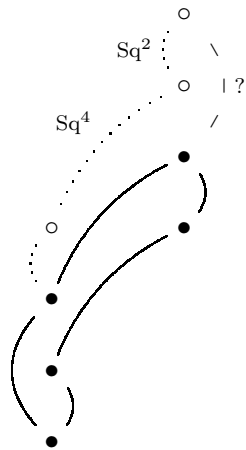
\[
\xymatrix@C=0.5cm@R=0.5cm{
&& \circ & \\
&& \circ\ar@/^8pt/@{.}[u]^{\Sq^2} & \\
&& \bullet\ar@{--}@/_15pt/[uu]_{?} & \\
\circ\ar@{.}@/^8pt/[rruu]^{\Sq^4}&&\bullet\ar@/_8pt/@{-}[u]& \\
\bullet\ar@/^8pt/@{.}[u]\ar@{-}@/^8pt/[rruu]&&& \\
\bullet\ar@{-}@/^8pt/[rruu]&&& \\
\bullet\ar@{-}@/^15pt/[uu]\ar@/_8pt/@{-}[u]&&& \\
}
\]
Using the type (B) Wall relation~\cite{CTCW:A-gens}
we see that $\Sq^4\Sq^4 + \Sq^2\Sq^4\Sq^2$
acts trivially on $\Joker(2)$ and so the
dashed line must be a non-trivial $\Sq^4$
action. Therefore $H^*(C(g))\iso\Delta\StA(1)$
as $\StA(2)$-modules. Of course there are
two possible $\StA$ actions depending on
which choice of $J(2)$ we make giving
different $\Sq^{16}$ actions.
\end{rem}

An attempt at a direct analogue of the
preceding argument for the next case runs
into difficulties as $\pi_{18}(J(3))\neq\{0\}$
when $J(3)$ is either of~$J(3)_0$ or~$J(3)_1$
(see Figures~\ref{fig:J(3)0} and~\ref{fig:J(3)1}).

\bigskip
In the other direction we have some
non-existence results.
\begin{thm}\label{thm:NoJokers}
For $n\geq4$, there is no finite $2$-local
CW~spectrum whose cohomology restricts to an
$\StA(n)$-module isomorphic to\/ $\Joker(n)$.
\end{thm}

When $n\geq5$ such a spectrum would violate
Adams' Hopf invariant~$1$ theorem because
of the large gap between the two elements
of lowest degrees. However, in all cases
we can use the precise statement of the
following crucial result on the factorisation
of primary operations. Here~$X$ is a
connective spectrum and we explain the
notation after the statement.
\begin{thm}
[Adams~\cite{JFA:HopfInvt1}*{theorem~4.6.1}]
\label{thm:JFA}
Let $k\geq3$ and suppose that $u\in H^m(X)$
for $m>0$ satisfies $\Sq^{2^r}u=0$
for $0\leq r\leq k$. Then
\[
\Sq^{2^{k+1}}u \equiv
\sum_{\substack{0\leq i\leq j\leq k \\ j\neq i+1}}
\alpha_{i,j,k}\Phi_{i,j}u
\pmod{\mathrm{indeterminacy}}.
\]
\end{thm}

In this result, the secondary operation
$\Phi_{i,j}$ has degree $2^i+2^j-1$, and
the primary operation $\alpha_{i,j,k}\in\StA$
has degree $2^{k+1}-2^i-2^j+1$. The
indeterminacy is the sum of the indeterminacies
of all the $\Phi_{i,j}$ appearing and
has form
\begin{equation}\label{eq:Adamsindeterminacy}
\sum_{\substack{0\leq i\leq j\leq k \\ j\neq i+1}}
\alpha_{i,j,k}Q^*(X;i,j)
\end{equation}
for certain subgroups $Q^*(X;i,j)\subseteq H^*(X)$.
Finally, for each pair $i,j$ occurring,
\begin{equation}\label{eq:Adamsindeterminacy-degalpha}
1\leq \deg\alpha_{i,j,k} \leq 2^{k+1}-1.
\end{equation}

\begin{proof}
[Proof of\/ \emph{Theorem~\ref{thm:NoJokers}}]
Let $n\geq4$ and suppose that a Joker spectrum
$J$ exists for this~$n$.
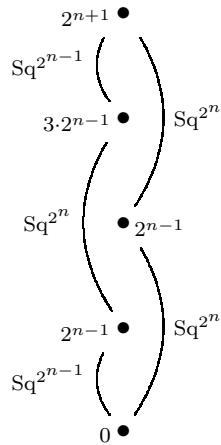
\[
\xymatrix{
& {}_{2^{n+1}\,}\bullet_{\,\ph{2^{n+1}}}
\ar@/_10pt/@{-}[d]_{\Sq^{2^{n-1}}}\ar@/^15pt/@{-}[dd]^{\Sq^{2^{n}}} & \\
& {}_{3\.2^{n-1}\,}\bullet_{\,\ph{3\.2^{n-1}}}
\ar@/_15pt/@{-}[dd]_{\Sq^{2^{n}}} & \\
& {}_{\ph{2^{n-1}}\,}\bullet_{\,2^{n-1}}
\ar@/^15pt/@{-}[dd]^{\Sq^{2^{n}}} & \\
& {}_{2^{n-1}\,}\bullet_{\,\ph{2^{n-1}}}
\ar@/_10pt/@{-}[d]_{\Sq^{2^{n-1}}} & \\
& {}_{0\,}\bullet_{\,\ph{0}} & \\
}
\]
Consider the non-zero element $u$ in
degree $2^{n-1}$. Taking $k=n-1$, we
can apply Theorem~\ref{thm:JFA}.
Carefully examining the possible terms
in the sum we find that they are all~$0$,
and similarly so is the indeterminacy.
The conclusion is that $\Sq^{2^n}u=0$,
contradicting the assumptions on~$J$.
\end{proof}

\section{Some unstable realisations}\label{sec:Unstable}

Now we turn to the question of unstable
realisations, i.e., as the cohomology of
a space. If~$X$ is a $2$-local space whose
cohomology $\tilde{H}^*(X)$ is isomorphic
to $\Joker^*[n]$ as an $\StA(1)$-module
then~$n\geq2$ since $\Sq^2$ acts non-trivially
on the bottom generator. Similarly, realising
the $\StA$-module $\Joker_1[n]$ unstably
requires that~$n\geq4$.
\begin{thm}\label{thm:Joker-Unstable}
There are finite $2$-local CW~complexes~$X_2$
and $X_4$ such that as\/ $\StA$-modules,
\[
\tilde{H}^*(X_2) \iso \Joker_0^*[2],
\quad
\tilde{H}^*(X_4) \iso \Joker_1^*[4].
\]
\end{thm}
\begin{proof}
Corollary~\ref{cor:Jokerspectra} suggests
looking for an unstable realisation of
the Joker in the space
$\underline{\kO\langle2\rangle}_0=B\SO$.
However, the cohomology of this is too
large in low degrees, instead we look
at $B\SO(3)$.

Recall the Wu formula
\begin{equation}\label{eq:Wu}
\Sq^rw_m =
w_rw_m + \sum_{1\leq i\leq r}\binom{r-m}{i}w_{r-i}w_{m+i}.
\end{equation}
Using this, in $H^*(B\SO(3)) = \F_2[w_2,w_3]$
we obtain
\[
\Sq^1w_2 = w_3,
\quad
\Sq^2w_3 = 0.
\]
Thus we obtain a copy of the Joker in
the $\StA(1)$-module $H^*(B\SO(3))$.
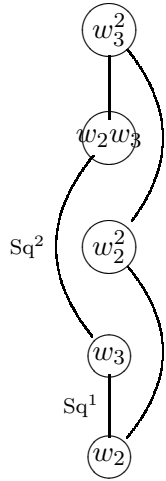
\[
\xymatrix{
&*+[o][F]{w_3^2}\ar@/^20pt/@{-}[dd]^{}\ar@{-}[d]_{}& \\
&*+[o][F]{w_2w_3}\ar@/_20pt/@{-}[dd]_{\Sq^2}& \\
&*+[o][F]{w_2^2}\ar@/^20pt/@{-}[dd]^{}& \\
&*+[o][F]{w_3}\ar@{-}[d]_{\Sq^1}& \\
&*+[o][F]{w_2}&
}
\]
However, $H^6(B\SO(3))=\F_2\{w_3^2,w_2^3\}$,
so we next remove the additional generator
by considering the fibre of the map
classifying $w_2^3$,
\[
B\SO(3) \xrightarrow{w_2^3} K(\F_2,6)
\]
which we will denote by $B\SO(3)\{w_2^3\}$.
Calculating $H^*(B\SO(3)\{w_2^3\})$ using
the Serre spectral sequence or the
Eilenberg-Moore spectral sequence we find
that when~$k\leq 6$,
\[
H^k(B\SO(3)\{w_2^3\})
\iso (\F_2[w_2,w_3]/(w_2^3))^k,
\]
so taking the $6$-skeleton of a minimal
CW~realisation (in the sense
of~\cite{AJB&JPM}*{section~3} for example)
we obtain an isomorphism of $\StA(1)$-modules
\[
H^*(B\SO(3)\{w_2^3\}^{[6]}) \iso \Joker_0[2].
\]

To realise $\Joker_1[4]$, we start with
an unstable complex $S^3\cup_{\eta_3}e^5\cup_2e^6$
which exists since the suspension of the
Hopf map $S^3\ra S^2$ gives an element
$\eta_3\in\pi_4(S^3)$ of order~$2$,
see~\cite{Toda}*{chapter~V}. Smashing
with the Moore space $S^1\cup_2e^2$ we
obtain a CW~complex~$X'$ whose cohomology
as an $\StA(1)$-module realises the
$4$-fold suspensions of the 'whiskered
Joker' module
\[
\Joker'
= \StA(1)/\StA(1)\{\Sq^2\Sq^1\Sq^2\}
= \StA(1)/\StA(1)\{\Sq^2\Sq^3\}.
\quad
\xymatrix{
&\bullet\ar@/^20pt/@{-}[dd]^{}\ar@{-}[d]_{}& \\
&\bullet\ar@/_20pt/@{-}[dd]_{\Sq^2}&\bullet \\
&\bullet\ar@/^20pt/@{-}[dd]^{}\ar@{-}[ur]& \\
&\bullet\ar@{-}[d]_{\Sq^1}& \\
&\bullet&
}
\]
Labelling cells and cohomology generators
in the obvious way, $X'$ has the following
cell diagram.

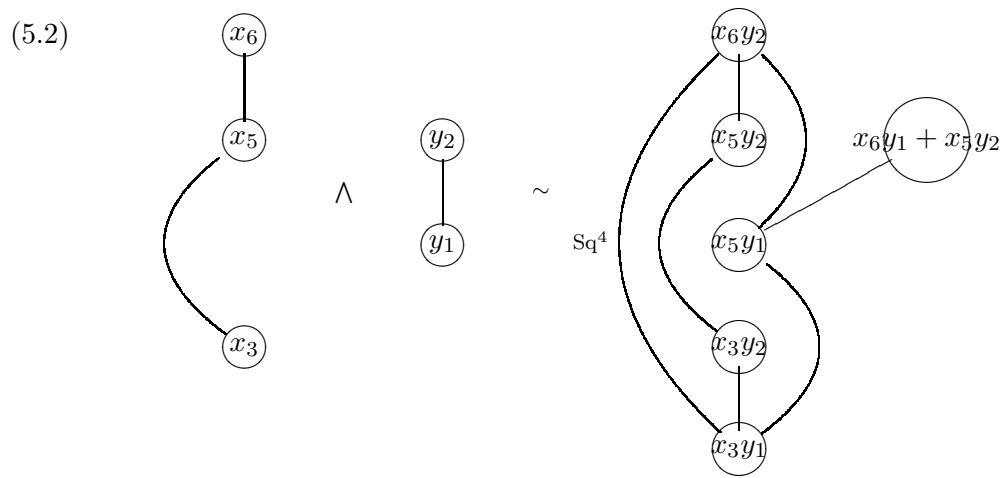
\begin{equation}\label{eq:Joker'}
\xymatrix{
*+[o][F]{x_6} &&&&&*+[o][F]{x_6y_2}& \\
*+[o][F]{x_5}\ar@{-}[u]&\ar@{}[d]|(.5){\bigwedge}
&*+[o][F]{y_2}&\ar@{}[d]|(.5){\sim}&&*+[o][F]{x_5y_2}\ar@{-}[u]
&*+[o][F]{x_6y_1+x_5y_2} \\
&&*+[o][F]{y_1}\ar@{-}[u]&&&*+[o][F]{x_5y_1}\ar@{-}@/_25pt/[uu]\ar@{-}[ur]& \\
*+[o][F]{x_3}\ar@{-}@/^30pt/[uu]&&&&&*+[o][F]{x_3y_2}\ar@{-}@/^30pt/[uu]& \\
&&&&&*+[o][F]{x_3y_1}\ar@{-}[u]\ar@{-}@/_30pt/[uu]
\ar@{-}@/^45pt/[uuuu]^{\Sq^4}& \\
}
\end{equation}
Notice that as well as the $\Sq^1$ and
$\Sq^2$ actions we also have
$\Sq^4(x_3y_1)=x_6y_2$ so this agrees
with $\Joker'_1[4]$ as an $\StA$-module.

%
%

We will begin by showing that there
is a factorisation
\[
\xymatrix{
X'\ar[r]\ar@/^15pt/[rr]
&\underline{\kO}_7\ar[r]
& K(\F_2,7)
}
\]
of the map classifying $x_6y_1+x_5y_2\in H^7(X')$.
We will do this by producing a map of
spectra $\Sigma^\infty X'\ra\Sigma^7\kO$
by dualising a map
\[
S^1\ra\kO\wedge\Sigma^8D\Sigma^\infty X'
\sim \kO\wedge W'',
\]
where $W''$ is a CW spectrum whose
cohomology realises the other whiskered
Joker $\StA(1)$-module $\Joker_1''$
shown in the following diagram.
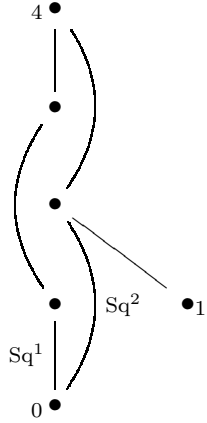
\[
\xymatrix{
&& {}_{4\,}\bullet_{\,\ph{4}}
\ar@{-}[d]\ar@/^15pt/@{-}[dd] & \\
&& \bullet
\ar@/_15pt/@{-}[dd] &  \\
 && \bullet
\ar@/^15pt/@{-}[dd]^{\Sq^2} & \\
&& \bullet
\ar@{-}[d]_{\Sq^1} & \ar@{-}[ul]{}_{\,\ph{1}}\bullet_{1\,}\\
&& {}_{0\,}\bullet_{\,\ph{0}} & \\
}
\]
Since we are interested in elements
of $\pi_1(\kO\wedge W'')$, we need
to consider the $t-s=1$ column in
the Adams spectral sequence
\[
\mathrm{E}_2^{s,t}
= \Ext_{\StA}^{s,t}(H^*(\kO\wedge W''),\F_2)
\iso \Ext_{\StA(1)}^{s,t}(H^*(W''),\F_2)
\Lra \pi_{t-s}(\kO\wedge W'')
\]
and a portion of the $\mathrm{E}_2$-term
is shown in Figure~\ref{fig:DJ''}. As the
generator in $\mathrm{E}_2^{0,1}$ cannot
support a differential there is a non-trivial
element of $\pi_1(\kO\wedge W'')$ detected
in the zero line by the only
$\StA(1)$-indecomposable element of $H^1(W'')$.
Hence there is a dual element of $\kO^7(X')$
with the desired properties.

Now take a minimal CW~complex equivalent to
the fibre of the map $X'\ra\underline{\kO}_7$
and let~$X_4$ be its $8$-skeleton. By a
straightforward calculation with either of
the Serre or Eilenberg-Moore spectral sequences
and making use of the~$\kO$ results of
Examples~\ref{examps:Conn-kOtmf}, we find
that~$H^*(X)$ realises the $\StA$-module~$\Joker_1[4]$.
\end{proof}
\begin{thm}\label{thm:Joker(2)-Unstable}
There are finite $2$-local CW~complexes~$Y_4$
and $Y_8$ such that as\/ $\StA$-modules,
\[
\tilde{H}^*(Y_4) \iso \Joker(2)_0^*[4],
\quad
\tilde{H}^*(Y_8) \iso \Joker(2)_1^*[8].
\]
\end{thm}

\begin{proof}
A similar construction to that of $X_2$
starting with $B\SU(3)$ leads to an
unstable realisation of $\Joker(2)_0[4]$.

We will realise $\Joker(2)_1[8]$ using
a similar approach to that for~$X_4$.
By Toda~\cite{Toda}*{proposition~5.8},
$\pi_{10}(S^6)$ is trivial, so the
suspensions of the Hopf maps give
elements $\eta_9\in\pi_{10}(S^9)$ and
$\nu_6\in\pi_9(S^6)$ which satisfy
\[
0 = \nu_6\circ\eta_9\in\pi_{10}(S^6).
\]
Hence we can form
$S^5\cup_{\nu_5}e^9\cup_{\eta_9}e^{11}$
and $S^3\cup_{\eta_3}e^5$. By smashing
these together we obtain a CW~complex
\[
Y' =
(S^5\cup_{\nu_5}e^9\cup_{\eta_9}e^{11})
\wedge
(S^3\cup_{\eta_3}e^5)
\]
whose cohomology realises the $\StA$-module
with non-trivial $\Sq^8$-action and is
the $8$-fold suspension of the whiskered
double Joker cyclic $\StA(2)$-module
\[
\Joker(2)_1' =
\StA(2)/\StA(2)
\{\mathrm{P}_1^0,\mathrm{P}_1^1,\mathrm{P}_2^1,\Sq^4\Sq^6\}.
\]
\[
\xymatrix{
&& {}_{16\,}\bullet_{\,\ph{16}}
\ar@/_10pt/@{-}[d]\ar@/^15pt/@{-}[dd] & \\
&& \bullet
\ar@/_15pt/@{-}[dd] & \ar@/^10pt/@{-}[dl]{}_{\,\ph{14}}\bullet_{14\,} \\
 \Joker(2)_1'[8] && \bullet
\ar@/^15pt/@{-}[dd]^{\Sq^4} & \\
&& \bullet
\ar@/_10pt/@{-}[d]_{\Sq^{2}} & \\
&& {}_{8\,}\bullet_{\,\ph{0}} & \\
}
\]
We would like to define a map
$Y'\ra\underline{\tmf}_{14}$ so that
the cohomology class in $H^{14}(\underline{\tmf}_{14})$
carried on the bottom cell is mapped to $\Sq^2\Sq^4y_8$
by the induced homomorphism, where $y_8\in H^8(Y)$
is the generator. Such a map corresponds to a map
of spectra $\Sigma^\infty Y'\ra\Sigma^{14}\tmf$
or equivalently a map
\[
S^0\ra\Sigma^{14}(D\Sigma^\infty Y')\wedge\tmf
\sim \Sigma^{-2}Z''\wedge\tmf
\]
where $Z''$ is a CW spectrum whose cohomology
realises the other whiskered double Joker
$\StA(2)$-module $\Joker(2)_1''$ shown in
the following diagram.
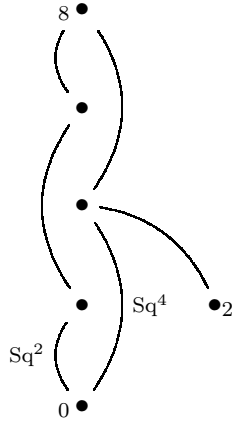
\[
\xymatrix{
&& {}_{8\,}\bullet_{\,\ph{8}}
\ar@/_10pt/@{-}[d]\ar@/^15pt/@{-}[dd] & \\
&& \bullet
\ar@/_15pt/@{-}[dd] &  \\
 && \bullet
\ar@/^15pt/@{-}[dd]^{\Sq^4} & \\
&& \bullet
\ar@/_10pt/@{-}[d]_{\Sq^{2}} & \ar@/_10pt/@{-}[ul]{}_{\,\ph{2}}\bullet_{2\,}\\
&& {}_{0\,}\bullet_{\,\ph{0}} & \\
}
\]
Since we are interested in elements
of $\pi_2(\tmf\wedge Z'')$, we need
to consider the $t-s=2$ column in
the Adams spectral sequence
\[
\mathrm{E}_2^{s,t}
= \Ext_{\StA}^{s,t}(H^*(\tmf\wedge Z''),\F_2)
\iso \Ext_{\StA(2)}^{s,t}(H^*(Z''),\F_2)
\Lra \pi_{t-s}(\tmf\wedge Z'')
\]
and a portion of the $\mathrm{E}_2$-term
is shown in Figure~\ref{fig:DJ(2)''}.
As the generator in $\mathrm{E}_2^{0,2}$
cannot support a differential this
shows that there is a suitable element
of $\pi_2(\tmf\wedge Z'')$ and hence
of $\tmf^{14}(Y')$.

Now consider the fibre of the above map
$Y'\ra\underline{\tmf}_{14}$. By a spectral
sequence calculation and making use of
the~$\tmf$ results of Examples~\ref{examps:Conn-kOtmf}
we see that its cohomology agrees with
$\Joker(2)_1[8]$ up to degree~$21$. The
$16$-skeleton of a minimal CW~realisation
of this fibre is a CW complex~$Y_8$ whose
cohomology as an $\StA$-module agrees
with~$\Joker(2)_1[8]$.
\end{proof}

It is not clear how to realise $\Joker(3)_0[8]$
since there is no obvious analogue of
$B\SO(3)$ and $B\SU(3)$ which appears
relevant. Similarly, our argument for
$X_4$ and $Y_8$ has no obvious
generalisation since because of the
non-existence of suitable elements
of Hopf invariant~$1$ there is no
spectrum playing an analogous r\^ole
to~$\kO$ and~$\tmf$ in the last steps.

\bigskip
\section*{Concluding remarks}

The appearance of the elements of Kervaire
invariant~$1$ in our realisations of Joker
modules raises the question of whether there
other $\StA(n)$-modules which admit realisations
when $\theta_n$ exists, i.e., when $n=4$,~$5$
and possibly~$6$. In particular,
by~\cite{ZX:StrongKervaire62}*{theorem~5.2},
\[
\{\theta_4\} =
\langle2,\sigma^2+\kappa,2\sigma,\sigma\rangle,
\]
while older work of Barratt, Mahowald \& Tangora
and Kochman shows that
\[
\{\theta_4\}
= \langle2,\sigma^2,2,\sigma^2\rangle
= \langle2,\sigma^2,\sigma^2,2\rangle
= \langle2\sigma,\sigma,2\sigma,\sigma\rangle
= \langle2,\sigma^2,2\sigma,\sigma\rangle.
\]
These suggest the intriguing possibility
that appropriate constructions associated
with such $4$-fold Toda brackets might
lead to realisation results for some
interesting~$\StA(4)$-modules.

\newpage
\bigskip
\appendix
\section{Doubling}\label{app:Doubling}

Doubling for $\StA$ and $\StA(n)$ are
discussed in Margolis~\cite{HRM:Book}*{section~15.3}
(for a particularly relevant result on
modules see theorem~31). Rather than
follow Margolis directly, we will give
an account of doubling emphasising the
dual situation.

The dual of~$\StA(n)$ is the quotient
Hopf algebra
\begin{align*}
\StA(n)_*
&=
\StA_*/(\zeta_1^{2^{n+1}},\zeta_2^{2^{n}},
\ldots,\zeta_{n+1}^{2},\zeta_{n+2},\ldots) \\
&=
\F_2[\bar{\zeta_1},\bar{\zeta_2},\ldots,\bar{\zeta_{n+1}}]
/(\bar{\zeta_1}^{2^{n+1}},\bar{\zeta_2}^{2^{n}},\ldots,\bar{\zeta_{n+1}}^{2}),
\end{align*}
where $\bar{(-)}$ indicates residue class.
The dual of the normal exterior subHopf
algebra $\StE(n)\subseteq\StA(n)$ generated
by the Milnor primitives $\mathrm{P}^0_t$
($1\leq t \leq n+1$) is the quotient
exterior algebra
\[
\StE(n)_*
=
\StA_*/(\zeta_1^2,\zeta_2^2,\ldots,
\zeta_{n+1}^{2},\zeta_{n+2},\zeta_{n+3},\ldots)
\iso
\StA(n)_*/(\bar{\zeta_1}^2,\bar{\zeta_2}^2,\ldots,\bar{\zeta_{n+1}}^2).
\]
The dual of the quotient Hopf algebra
\[
\StA(n)/\!/\StE(n) = \StA(n)\otimes_{\StE(n)}\F_2
\iso \StA(n)/\StA(n)\StE(n)^+
\]
is
\[
\StA(n)_*\square_{\StE(n)_*}\F_2 =
\F_2[\bar{\zeta_1}^2,\bar{\zeta_2}^2,\ldots,\bar{\zeta_{n+1}}^2]
/(\bar{\zeta_1}^{2^{n+1}},\bar{\zeta_2}^{2^{n}},\ldots,\bar{\zeta_{n+1}}^{2})
\subseteq\StA(n)_*.
\]

The dual of doubling is the Frobenius
homomorphism.
\[
\xymatrix{
\StA(n)_*\ar[dr]_{\mathbf{f}}\ar[r]
 & \StA(n+1)_*\square_{\StE(n+1)_*}\F_2\ar[d] \\
& \StA(n+1)_* \\
\bar{\zeta_r}\ar@{|->}[r]& \bar{\zeta_r}^2
}
\]
There are obvious iterations of this, namely
$\mathbf{f}^{(k)}$ where $k\geq0$,
\[
\xymatrix{
\StA(n)_*\ar[dr]_{\mathbf{f}^{(k)}}\ar[r]
 & \StA(n+k)_*\square_{\StE(n+k)_*}\F_2\ar[d] \\
& \StA(n+k)_* \\
\bar{\zeta_r}\ar@{|->}[r]& \bar{\zeta_r}^{2^k}
}
\]
so that $\mathbf{f}^{(0)}=\mathbf{f}$. Each
$\mathbf{f}^{(k)}$ is clearly a Hopf algebra
homomorphism and dually there is a Verschiebung
homomorphism $\mathbf{v}^{(k)}\:\StA(n+k)\to\StA(n)$
which is also a Hopf algebra homomorphism.
\begin{equation}\label{eq:versch-k}
\xymatrix{
\StA(n)\ar@{<-}[dr]_{\mathbf{v}^{(k)}}\ar@{<-}[r]
 & \StA(n+k)/\!/\StE(n+k)\ar@{<-}[d] \\
& \StA(n+k)
}
\end{equation}
The effect of $\mathbf{v}^{(k)}$ is easily
seen using the Milnor basis dual to the
monomial basis in the elements
$\bar{\xi_r}=\chi\bar{\zeta_r}$. Since
\[
\mathbf{f}^{(k)}(\bar{\xi_1^{i_1}\cdots\xi_\ell^{i_\ell}})
=
\bar{\xi_1^{2^ki_1}\cdots\xi_\ell^{2^ki_\ell}}
\]
we have
\begin{equation}\label{eq:Verschiebung}
\mathbf{v}^{(k)}(\Sq(j_1,\ldots,j_\ell)) =
\begin{dcases*}
\Sq(j'_1,\ldots,j'_\ell)
& if $j_r=2^kj'_r$ for all $r$, \\
0 & otherwise.
\end{dcases*}
\end{equation}
Since $\Sq(j)=\Sq^j$,
\begin{equation}\label{eq:Verschiebung-Sq}
\mathbf{v}^{(k)}(\Sq^j) =
\begin{dcases*}
\Sq^{j'} & if $j=2^kj'$, \\
0 & otherwise.
\end{dcases*}
\end{equation}
Finally, the elements
\[
\mathbf{v}^{(k)}(\mathrm{P}^s_t) =
\Sq(\overset{t}{\overbrace{0,\ldots,0,2^s}})\in\StA(n)
  \quad(1\leq t\leq n+1)
\]
satisfy
\begin{equation}\label{eq:Verschiebung-P}
\mathbf{v}^{(k)}(\mathrm{P}^{s}_t) =
\begin{dcases*}
\mathrm{P}^{s-k}_t & if $s\geq k$, \\
0 & otherwise.
\end{dcases*}
\end{equation}

Let $\leftidx{_{\StA(n)}}{\mathcal{M}}{^{(r)}}$
be the full subcategory of $\leftidx{_{\StA(n)}}{\mathcal{M}}{}$
consisting of modules concentrated in degrees
divisible by~$2^r$, so for example
$\leftidx{_{\StA(n)}}{\mathcal{M}}{^{(0)}}
= \leftidx{_{\StA(n)}}{\mathcal{M}}{}$,
and
$\leftidx{_{\StA(n)}}{\mathcal{M}}{^{(1)}}
= \leftidx{_{\StA(n)}}{\mathcal{M}}{^\ev}$.
The Verschiebung $\mathbf{v}^{(k)}$ together
with the quotient homomorphism
$\pi\:\StA(n+k)\to\StA(n+k)/\!/\StE(n+k)$
induces restriction functors between categories
of left modules fitting into the following
commutative diagram.
\begin{equation}\label{eq:versch-kModules}
\xymatrix{
\leftidx{_{\StA(n)}}{\mathcal{M}}{}\ar[drr]_{\Delta^{(k)}} 
\ar[rr]^(.4){(\mathbf{v}^{(k)})^*}
&&
\leftidx{_{\StA(n+k)/\!/\StE(n+k)}}{\mathcal{M}}{^{(k)}}\ar[d]^{\pi^*} \\
&& \leftidx{_{\StA(n+k)}}{\mathcal{M}}{^{(k)}} \\
}
\end{equation}
Here $\Delta^{(k)}$ multiples degrees by~$2^k$
and it is a monoidal functor since $\mathbf{v}^{(k)}$
and $\pi$ are both homomorphisms of Hopf
algebras.
\begin{rem*}
By~\cite{HRM:Book}*{theorem~15.3.31} $\Delta^{(1)}$
is an isomorphism of categories, but when $k>1$
this is not true. However, this can be corrected
by replacing $\StA(n+k)/\!/\StE(n+k)$ by the quotient
of $\StA(n+k)$ by the ideal generated by a larger
set of the elements~$\mathrm{P}^s_t$. Let
\[
\StE(n+k,k) =
\F_2(\mathrm{P}^s_t : 1\leq t \leq n+k+1,\; 0\leq s <k)
\subseteq \StA(n+k)
\]
and consider the normal quotient
\[
\StA(n+k)/\!/\StE(n+k,k) = \StA(n+k)/\StA(n+k)\StE(n+k,k)^+
\iso \StA(n+k)\otimes_{\StE(n+k,k)}\F_2.
\]
Then \eqref{eq:versch-kModules} can be replaced
by
\begin{equation}\label{eq:versch-kModules2}
\xymatrix{
\leftidx{_{\StA(n)}}{\mathcal{M}}{}\ar[drr]_{\Delta^{(k)}} 
\ar[rr]^(.4){(\mathbf{v}^{(k)})^*}_(.4){\iso}
&&
\leftidx{_{\StA(n+k)/\!/\StE(n+k,k)}}{\mathcal{M}}{^{(k)}}\ar[d]^{\pi^*}_{\iso} \\
&& \leftidx{_{\StA(n+k)}}{\mathcal{M}}{^{(k)}}
}
\end{equation}
so that the proof of Margolis still applies
to show that this $\Delta^{(k)}$ is an
isomorphism of categories.

It seems worth remarking that $\Delta^{(k)}$
does not induce a functor on stable module
categories since it does not preserve
projective modules.
\end{rem*}

\section{Some connectivity results}\label{app:Connectivity}

Let $p$ be a prime and let $f\:X\to Y$ be a map
between two finite type $p$-local connective
spectra or spaces which are simply connected
or at least have abelian fundamental groups.

Recall that~$f$ is called an \emph{$n$-equivalence}
if $f_*\:\pi_k(X)\to\pi_k(Y)$ is an isomorphism
for $k<n$ and an epimorphism if $k=n$; this
is equivalent to the mapping cone~$C_f$ being
$n$-connected. It is well-known that the following
are also equivalent conditions:
\begin{itemize}
\item
$f_*\:H_k(X;\Z_{(p)})\to H_k(Y;\Z_{(p)})$
is an isomorphism if $k<n$ and an epimorphism
if $k=n$.
\item
$f_*\:H_k(X;\F_p)\to H_k(Y;\F_p)$ is an
isomorphism if $k<n$ and an epimorphism
if $k=n$.
\end{itemize}

The next result relates connectivity information
for spectra and their associated infinite loop
spaces. Although such results are undoubtedly
standard we are not aware of convenient references
and we use them to establish
Examples~\ref{examps:Conn-kOtmf}.

We will denote the $m$-th space in a spectrum~$X$
by $\underline{X}_m=\Omega^\infty\Sigma^mX$.
If  $f\:X\ra Y$ is a map of $(-1)$-connected
spectra then for each $m\geq0$ there is an
induced infinite loop map
$f_m\:\underline{X}_m\ra\underline{Y}_m$.
\begin{lem}\label{lem:Spectra-InfLoopSpce}
Let $f\:X\ra Y$ be an $n$-equivalence. Then
for each\/ $m\geq1$,
$f_m\:\underline{X}_m\ra\underline{Y}_m$
is an $m+n$-equivalence, hence
$(f_m)_*\:H_k(\underline{X}_m;\F_p)
    \to H_k(\underline{Y}_m;\F_p)$
is an isomorphism if $k<m+n$ and an epimorphism
if $k=m+n$.
\end{lem}

Here is a sample application; we only state
this for the prime~$2$, but a similar result
also holds for odd primes.
\begin{cor}\label{cor:Conn-Unit}
Take $p=2$ and let $X$ be a $(-1)$-connected
spectrum and suppose that $\pi_0(X)$ is
a cyclic $\Z_{(2)}$-module with generator
given by a map $j\:S^0\ra X$. If $j$ is
an $n$-equivalence then for each $m>n$,
and $m<k\leq m+n$,
\[
H_k(\underline{X}_m;\F_2) = 0.
\]
\end{cor}
\begin{proof}
Recall that for $m\geq1$, the homology of
$\underline{S^0}_m=\dlQ S^m$ is given by
\[
H_*(\dlQ S^m;\F_2) =
\F_2[\dlQ^Ix_m : \text{$I$ admissible, $\exc(I)>m$}],
\]
where $x_m\in H_m(\dlQ S^m;\F_2)$. Thus
the three elements of lowest degree are
$x_m,x_m^2,\dlQ^{2m+1}x_m$ in degrees~$m$,
$2m>m+n$ and $2m+1>m+n$ respectively.

The infinite loop map $j_m$ induces an
algebra homomorphism $(j_m)_*$ over the
Dyer-Lashof algebra. By assumption on~$j$,
\[
(j_m)_*\:H_k(\dlQ S^m;\F_2)
             \ra H_k(\underline{X}_m;\F_2)
\]
is an isomorphism when $k<m+n$ and
an epimorphism when $k=m+n$.
\end{proof}

Thus the lowest degree element of
$H_*(\underline{X}_m;\F_2)$ not in
the image of $(j_m)_*$ occurs in some
degree~$k_0\geq m+n+1$ where $k_0-m$
is also the smallest degree for which
\[
\coker[j_*\:H_k(S^0;\F_2)\ra H_k(X;\F_2)]
\neq0.
\]
\begin{examps}\label{examps:Conn-kOtmf}
We set $H_*(-)=H_*(-;\F_2)$.

Recall that
\[
H_*(\kO)
= \F_2[\zeta_1^4,\zeta_2^2,\zeta_3,\ldots]\subseteq\StA_*,
\quad
H_*(\tmf)
= \F_2[\zeta_1^8,\zeta_2^4,\zeta_3^2,\zeta_4,\ldots]\subseteq\StA_*.
\]
Thus on $H_*(-)$ the units $j^\kO\:S\ra\kO$
and $j^\tmf\:S\ra\tmf$ induce homomorphisms
whose cokernels $\coker j^\kO_*$ and
$\coker j^\kO_*$ have non-zero elements
of lowest degrees~$4$ and~$8$ respectively.
Thus $j^\kO$ is a $3$-equivalence and
$j^\tmf$ is a $7$-equivalence.

For $4\leq m<k\leq m+3$,
\[
H_k(\underline{\kO}_m) = 0,
\]
while for $8\leq m<k\leq m+7$,
\[
H_k(\underline{\tmf}_m) = 0.
\]
\end{examps}

The results for $\kO$ can also be deduced from
work of Dena Cowen~Morton~\cite{DSCM:Hopfringbo}
but as far as we know the algebra structure of
the Hopf ring for~$\tmf$ has not been determined.

\section{Some $\Ext$ charts}\label{app:Charts}

These charts were produced using Bob Bruner's
programmes available at the following address.
\begin{center}
\href{http://www.math.wayne.edu/~rrb/cohom/index.html}
{http://www.math.wayne.edu/$\sim$rrb/cohom/index.html}
\end{center}

\begin{figure}[htb]
\centering


\caption{$\Ext_{\StA}^{s,t}(\F_2,\F_2)$
for $0 \leq t-s \leq 20$, $0 \leq s \leq 12$.}
\label{fig:ASS-S0}
\end{figure}

\begin{figure}[ht!]
%
\caption{$\Ext_{\StA}^{s,n+s}(\Joker_0,\F_2)$:
$0 \leq n \leq 20$ and $0 \leq s \leq 10$.}
\label{fig:J0}
\end{figure}

\bigskip
\begin{figure}[h!]
%
\caption{$\Ext_{\StA}^{s,n+s}(\Joker_1,\F_2)$:
$0 \leq n \leq 20$ and $0 \leq s \leq 10$.}
\label{fig:J1}
\end{figure}

\bigskip
\begin{figure}[ht!]
%
\caption{$\Ext_{\StA}^{s,n+s}(\Joker(2)_0,\F_2)$:
$0 \leq n \leq 20$ and $0 \leq s \leq 10$.}
\label{fig:J(2)0}
\end{figure}

\bigskip
\begin{figure}[ht!]
%
\caption{$\Ext_{\StA}^{s,n+s}(\Joker(2)_1,\F_2)$:
$0 \leq n \leq 20$ and $0 \leq s \leq 10$.}
\label{fig:J(2)1}
\end{figure}

\bigskip
\begin{figure}[ht!]
%
\caption{$\Ext_{\StA}^{s,n+s}(\Joker(3)_0,\F_2)$:
$0 \leq n \leq 20$ and $0 \leq s \leq 10$.}
\label{fig:J(3)0}
\end{figure}

\bigskip
\begin{figure}[ht!]
%
\caption{$\Ext_{\StA(1)}^{s,n+s}(\Joker'',\F_2)$:
$0 \leq n \leq 4$ and $0 \leq s \leq 8$.}
\label{fig:DJ''}
\end{figure}

\bigskip
\begin{figure}[ht!]
%
\caption{$\Ext_{\StA(2)}^{s,n+s}(\Joker(2)'',\F_2)$:
$0 \leq n \leq 8$ and $0 \leq s \leq 8$.}
\label{fig:DJ(2)''}
\end{figure}

\end{document}